\documentclass{amsart}
\usepackage{graphicx, amsmath, amssymb, amsthm}
\usepackage{mathabx}
\usepackage[mathscr]{euscript}
\usepackage{tikz-cd}
\usepackage{hyperref}
\usepackage[colorinlistoftodos]{todonotes}
\usepackage{dsfont}

\usepackage[backend=biber, style=ieee-alphabetic, maxbibnames=99, doi=true,url=false,giveninits=true, hyperref]{biblatex}
\renewbibmacro{in:}{}

\numberwithin{equation}{section}

\newtheorem{theorem}{Theorem}[section]
\newtheorem{lemma}[theorem]{Lemma}
\newtheorem{corollary}[theorem]{Corollary}

\newtheorem{proposition}[theorem]{Proposition}

\newtheorem*{theorem*}{Theorem}
\newtheorem*{proposition*}{Proposition}

\theoremstyle{remark}
\newtheorem{remark}[theorem]{Remark}
\newtheorem{example}[theorem]{Example}

\newtheorem{reminder}[theorem]{Reminder}

\theoremstyle{definition}

\newtheorem{definition}[theorem]{Definition}

\newcommand{\E}{{\mathbb E}}
\newcommand{\F}{{\mathbb F}}
\newcommand{\T}{{\mathbb{T}}}
\newcommand{\N}{{\mathbb N}}

\newcommand{\mc}[1]{\mathscr{#1}}
\newcommand{\mb}[1]{\mathbb{#1}}
\newcommand{\mf}[1]{\mathfrak{#1}}
\newcommand{\md}[1]{\mathds{#1}}
\newcommand{\Alg}{\mathrm{Alg}}
\newcommand{\unit}{{\md 1}}
\newcommand{\from}{\leftarrow}
\newcommand{\co}{\colon\thinspace}

\newcommand{\aquot}{/\!/}
\renewcommand{\SS}{{\mathbb{S}}}

\DeclareMathOperator{\Tor}{Tor}
\DeclareMathOperator{\Map}{Map}
\DeclareMathOperator{\End}{End}

\newcommand{\hocolim}{\mathrm{hocolim}}

\newcommand{\timesover}[1]{\underset{#1}{\times}}
\newcommand{\smashover}[1]{{\underset{#1}{\wedge}}}

\DeclareMathOperator{\id}{id}
\DeclareMathOperator{\LMod}{LMod}
\DeclareMathOperator{\EZMod}{LAct}

\DeclareMathOperator{\Free}{Free}

\DeclareMathOperator{\Fun}{Fun}
\DeclareMathOperator{\TwArr}{Tw}
\DeclareMathOperator{\Ar}{Ar}

\DeclareMathOperator{\twobar}{Bar}

\DeclareMathOperator*{\holim}{holim}

\newcommand{\NSFSupport}[1]{This material is based upon work supported by the National Science Foundation under Grant No. {#1}}

\newcommand{\MAHNSF}{DMS-2105019}

\newcommand{\MAHAddress}{University of California Los Angeles, Los Angeles, CA 90095}
\newcommand{\MAHemail}{\tt{mikehill@math.ucla.edu}}

\newcommand{\TLAddress}{University of Minnesota, Minneapolis, MN 55455}
\newcommand{\TLemail}{\tt{tlawson@umn.edu}}

\title{$\E_k$-pushouts and $\E_{k+1}$-tensors}

\author[MAH]{Michael A.~Hill}
\address{\MAHAddress}
\email{\MAHemail}

\author[TL]{Tyler Lawson}
\address{\TLAddress}
\email{\TLemail}

\thanks{\NSFSupport{\MAHNSF}}

\begin{document}

\maketitle

\begin{abstract}
  We prove a general result that relates certain pushouts of $\E_k$-algebras to relative tensors over $\E_{k+1}$-algebras. Specializations include a number of established results on classifying spaces, resolutions of modules, and (co)homology theories for ring spectra. The main results apply when the category in question has centralizers.

  Among our applications, we show that certain quotients of the dual Steenrod algebra are realized as associative algebras over $H\F_p \wedge H\F_p$ by attaching single $\E_1$-algebra relation, generalizing previous work at the prime $2$. We also construct a filtered $\E_2$-algebra structure on the sphere spectrum, and the resulting spectral sequence for the stable homotopy groups of spheres has $E_1$-term isomorphic to a regrading of the $E_1$-term of the May spectral sequence.
\end{abstract}

%\tableofcontents

% fix to make listoftodos work with the amsart package
%\makeatletter
%\providecommand\@dotsep{5}
%\makeatother
%\listoftodos\relax

\section{Introduction}\label{sec:intro}

For commutative monoids in any symmetric monoidal category, pushouts are often straightforward to compute: the pushout of a diagram $B \from A \to R$ of commutative monoids is equivalent to a relative tensor $B \otimes_A R$. Moreover, the tensor product often passes nicely to a derived setting by replacing it with the two-sided bar construction $\twobar(B,A,R)$ that calculates the homotopy pushout. By contrast, computing pushouts and homotopy pushouts of associative monoids is more involved, because there is no easy rewriting procedure to separate terms involving $B$, $R$, and $A$. This becomes rapidly clear when attempting to apply results like the Seifert--van Kampen theorem or compute pushouts in the category of associative rings.

Homotopy-theoretically, associativity and commutativity are part of a very broad range of levels of commutativity. An $\E_0$-algebra has a unit, but no multiplication that it is the unit for; an $\E_1$-algebra has a multiplication that is associative, up to higher coherences; an $\E_2$-algebra has a multiplication with structure related to the braid axioms; and this hierarchy proceeds through higher and higher stages until $\E_\infty$-algebras, which are associative and commutative up to higher coherences. These first became prominent as part of a recognition principle for iterated loop spaces \cite{BoardmanVogtRecognition, MayLoopspaces}, but now play an important role in higher algebra.

Computing pushouts of $\E_k$-algebras is, by and large, a gigantic pain in the neck.

Our main result in this paper asserts that, in certain cases, homotopy pushouts in $\E_k$-algebras can be computed by a (derived) tensor product over an $\E_{k+1}$-algebra.

\begin{theorem*}[\ref{thm:Ekpushout}]%[Bar--Bear theorem]\label{thm:Ekpushout}
  Suppose that $\mc C$ is a presentable $\E_{k+1}$-monoidal $\infty$-category with monoidal structure $\otimes$, that $B$ is an $\E_{k+1}$-algebra in $\mc C$, $A \to B$ is a map of $\E_k$-algebras in $\mc C$, and that $\E A$ is an $\E_{k+1}$-algebra freely generated by $A$. Given $R$ any $\E_k$-algebra in the category $\LMod_B$ of left $B$-modules, the induced natural diagram
  \[
  \begin{tikzcd}
    B \otimes A \otimes R \ar[r,"m \otimes 1"] \ar[d,swap,"1 \otimes m"] &
    B \otimes R \ar[d, "m"] \\
    B \otimes R \ar[r, "m", swap] &
    B \otimes_{\E A} R
  \end{tikzcd}
  \]
  is a homotopy pushout diagram in the category of $\E_k$-algebras in $\LMod_{B}$.
\end{theorem*}

\begin{remark}
  We need to be able to discuss $\E_k$-algebras, and their homotopy pushouts, in the category of left modules over an $\E_{k+1}$-algebra. To some degree, this forces us to adopt a foundational setup that can handle such structure.

  Throughout this paper we will use quasicategories as a model for $\infty$-categories, using \cite{LurieHTT} and \cite{LurieHA} as references. However, let us informally unpack the hypotheses of the main theorem. Asking that $\mc C$ is an $\infty$-category, up to equivalence, is the same as asking to have spaces of maps between objects. Presentability of $\mc C$ asks that $\mc C$ has homotopy limits and colimits, and that $\mc C$ is generated under homotopy colimits by a set of well-behaved objects. Being a monoidal $\infty$-category asks that we have an essentially associative tensor product $\otimes$ on $\mc C$ that is compatible with mapping spaces; being monoidal presentable means that the operation $X \otimes Y$ has to preserve homotopy colimits in each variable $X$ and $Y$ separately. For this tensor product to be an $\E_{k+1}$-monoidal structure, it also has to come equipped with a large amount of coherence information expressing the degree to which $\otimes$ is commutative.

  The benefit of starting with these types of hypotheses is that they are weak enough to apply in a very wide variety of circumstances, such as the categories of spaces, spaces over a fixed $\E_{k+1}$-space, spectra, modules over a structured ring spectrum, graded or filtered versions of the same, and many others.
\end{remark}

\subsection{Actions}

This particular endeavor will pass through a number of technical results. Before we embark, we would like to spend some time mapping out the underlying structure of the proof.

The first observation is that, by base change, the general case of our main theorem will follow from the case where the algebra $B$ is the enveloping algebra $\E A$, so it suffices to show that there is a homotopy pushout diagram
\[
  \begin{tikzcd}
    \E A \otimes A \otimes R \ar[r,"m \otimes 1"] \ar[d,swap,"1 \otimes m"] &
    \E A \otimes R \ar[d, "m"] \\
    \E A \otimes R \ar[r, "m", swap] &
    R
  \end{tikzcd}
\]
of $\E_k$-algebras in $\LMod_{\E A}$.

The next reduction is to the associative case. The Dunn additivity theorem implies that $\E_{k+1}$-algebras are equivalent to $\E_1$-algebras in the category of $\E_k$-algebras \cite{DunnAdditivity}. The main result will follow if we can find a sufficiently general result covering the case $k=0$: one whose hypotheses will apply when $\mc C$ is a category of $\E_k$-algebras. Unfortunately, even if $\mc C$ is monoidal presentable, the category of $\E_k$-algebras in $\mc C$ is typically not. Even in ordinary algebra, for example, taking the tensor product $A \otimes B$ of two associative algebras typically does not preserve colimits of associative algebras in each variable. 

This leads us to understanding module structures over ``free'' algebras. If we have a monoidal category $\mc C$ and we have an $\E_0$-algebra $\unit \to X$, we can consider the \emph{enveloping algebra} $\E X$, an $\E_1$-algebra generated by it. In many circumstances, we can give a simpler description of left $\E X$-modules. Such simpler descriptions are automatic if $M$ has an endomorphism object $\End(M)$, because then unital maps $X \to \End(M)$ are equivalent to algebra maps $\E X \to \End(M)$ by the universal property of $\E X$. 

Unfortunately, in categories of $\E_k$-algebras  these endomorphism objects rarely exist, but there is a closely related concept called the \emph{centralizer} \cite[\S 5.3.1]{LurieHA}. In Section~\ref{sec:centers} we will show that, when $\mc C$ has centralizers, the category of left $\E X$-modules is equivalent to a category of objects with an action of $X$.

\begin{theorem*}[\ref{thm:EZequivalence}]
  Suppose that $\mc C$ is a monoidal $\infty$-category with pullbacks, that $\mc M$ is left-tensored over $\mc C$, and that $\mc M$ has centralizers in $\mc C$.
  
  If $X$ is an $\E_0$-algebra in $\mc C$ and $\E X$ is an enveloping algebra for $X$, then there is an equivalence of $\infty$-categories $\LMod_{\E X}(\mc M) \to \EZMod_X(\mc M)$ between left $\E X$-modules in $\mc M$ and objects with a unital left action map $\lambda\co X \otimes M \to M$.
\end{theorem*}

With this in mind, to show that we have a homotopy pushout in the category of left $\E X$ modules, it suffices to show that it is a homotopy pushout in the category $\EZMod_X$ of objects with an action of $X$. To do this, we need to be able to compute the space of maps between two objects with action maps. Section~\ref{sec:actions} is devoted to understanding this, and to proving the following result.

\begin{proposition*}[\ref{prop:EZpushout}]
  Suppose that the forgetful functor $\EZMod_X(\mc M) \to \mc M$ has a left adjoint $L$. Then there exists a natural transformation $\lambda\co L(X \otimes M) \to L(M)$, such that for any object $M$ with left action $\lambda_M\co X\otimes M \to M$, the diagram
  \[
    \begin{tikzcd}
      L(X \otimes M) \ar[r,"\lambda"] \ar[d,"L(\lambda_M)",swap] &
      LM \ar[d,"\epsilon"] \\
      LM \ar[r,"\epsilon",swap] &
      M
    \end{tikzcd}
  \]
  is a homotopy pushout in $\EZMod_X(\mc M)$.
\end{proposition*}

These results are proved with the aid of some results on computing centralizers, and mapping spaces between sections, using the twisted arrow category (Lemma~\ref{lem:sectionmaps} and Theorem~\ref{thm:diagramcentralizer}); these may be of some independent utility.

\subsection{Applications to stable homotopy}

Because of the broad range of cases where this result holds, it specializes to several topics of classical and more modern interest. Applied algebraically, this result recovers the Koszul resolution of a module over a tensor algebra. Applied to the category of pointed spaces, it recovers information about the James construction on a space. Applied to categories of of ring spectra, it recovers features important to Basterra and Mandell's study of homology and cohomology theories for $\E_k$ ring spectra. We will discuss several of these examples in Section~\ref{sec:applications}. However, our principal interest is in stable homotopy theory.

As one specific application, we can build on a description of the mod-$p$ Eilenberg--Mac Lane spectra $H\F_p$ as a Thom spectrum, due to Hopkins and Mahowald, to obtain the following alternative description.
\begin{theorem*}[\ref{thm:e3smash}]
  Fix a prime $p$, and consider the pair of $\E_3$-algebra maps
  \[
    \Free_{\E_3}(S^0) \rightrightarrows \mb S
  \]
  from the free $\E_3$-algebra on $S^0$, classifying $0$ and $p$ respectively. Then the (derived) smash product is the Eilenberg--Mac Lane spectrum $H\F_p$:
  \[
    H\F_p \simeq \mb S \smashover{\Free_{\E_3}(S^0)} \mb S
  \]
\end{theorem*}

This allows us to construct the following unusual filtration on the sphere spectrum. It is constructed by deliberately pushing the element $p \in \pi_0(\SS)$ into filtration $1$ in a category of filtered $\E_2$-algebras; this is inspired by techniques used by Baker \cite{BakerCloseEncounters} and Szymik \cite{SzymikCharacteristic}.

\begin{theorem*}[\ref{thm:weirdmayfilt}, \ref{cor:weirdmaysseq}]
  For any prime $p$, there exists a lift of the sphere spectrum to an $\E_2$-algebra in filtered spectra, giving rise to a spectral sequence of graded-commutative rings with abutment $\pi_* \mb S^\wedge_p$.
  
  For $p=2$, the $E_1$-term is a polynomial algebra
  \[
  \F_2[\mf h_{i,j}],
  \]
  where $\mf h_{i,j}$ are defined for $i \geq 1, j \geq 0$, with total degree and filtration given by the bidegree $(2^{i+j} - 2^j - 1, 2^{i+j-1})$.
  
  For $p$ odd, the $E_1$-term is
  \[
    \F_p[\mf v_i] \otimes \Lambda[\mf h_{i,j}] \otimes \F_p[\mf b_{i,j}],
  \]
  where the bidegree of $\mf v_i$ is $(2p^i-2, p^{i})$ for $i \geq 0$, of $\mf h_{i,j}$ is $(2p^{i+j}-2p^j-1,p^{i+j})$ for $i \geq 1, j \geq 0$, and of $\mf b_{i,j}$ is $(2p^{i+j+1}-2p^{j+1}-2,p^{i+j+1})$ for $i \geq 1, j \geq 0$.
\end{theorem*}

This $E_1$-term is isomorphic to a regrading of the $E_1$-term of the May spectral sequence, but converges directly to stable homotopy rather than the Adams $E_2$-term. The spectral sequence also has quite different behavior (see Remark~\ref{rmk:sseqblather}).

Our second application generalizes a result from \cite{BHLSZQuotient} to odd primes: taking quotients by ``central'' classes in an $\E_k$-algebra also kills $\E_{k+1}$-operations. It allows us to prove the following result, analyzing associative quotients of the dual Steenrod algebra. The case of $p=2$, where killing a generator in an $\E_1$-fashion precisely eliminated the free $\E_2$-algebra on it, was a mystery that helped us identify the main result of this paper.  
\begin{theorem*}[\ref{thm:oddsteenrod}]
  Let $p$ be an odd prime and $\mc A = H\F_p \wedge H\F_p$ be the commutative ring spectrum whose coefficient ring is the dual Steenrod algebra
  \[
    \mc A_* \cong \Lambda[\tau_0, \tau_1, \dots] \otimes \F_p[\xi_1, \xi_2, \dots].
  \]
  Then the free quotient $\mc A \aquot \tau_n$ in the category of associative $\mc A$-algebras has coefficient ring given by the quotient
  \[
    \mc A_\ast / (\xi_{n+1}, \xi_{n+2}, \dots, \tau_n, \tau_{n+1}, \dots),
  \]
  and $\mc A \aquot \bar \tau_n$ has coefficient ring given by the quotient
  \[
    \mc A_\ast / (\bar \xi_{n+1}, \bar \xi_{n+2}, \dots, \bar \tau_n, \bar \tau_{n+1}, \dots).
  \]
\end{theorem*}
This result relies implicitly on calculations of the Dyer--Lashof operations by Steinberger \cite{Hinfinity}.

\subsection{Acknowledgements}

The authors would like to thank 
Clark Barwick,
Rosona Eldred,
Rune Haugseng,
Piotr Pstr\c{a}gowski,
and
Dylan Wilson
for discussions related to this material. We also owe a significant debt to Jacob Lurie, who directed us to $\E_k$-centers and suggested that Proposition~\ref{prop:EZmaps} should be true.

This results in this paper originated in a joint project with Agn\`es Beaudry, Xiaolin Danny Shi, and Mingcong Zeng, and would not have been possible without their support.

\section{Functor categories}\label{sec:functors}

Our goal in this section is to collect together preliminary results about functor $\infty$-categories, with a particular emphasis on computing mapping spaces.

\begin{reminder}
  Recall that for an $\infty$-category $\mc C$, we have the \emph{arrow category} $\Ar(\mc C) = \Fun(\Delta^1, \mc C)$, and we have the \emph{twisted arrow category} $\TwArr(\mc C)$. In both cases, the objects are maps in $\mc C$. Morphisms $f \to g$ in $\Ar(\mc C)$ are commutative diagrams
  \[
    \begin{tikzcd}
      X \ar[r,"f"] \ar[d] & Y \ar[d] \\
      Z \ar[r,"g",swap] & W,
    \end{tikzcd}
  \]
  while morphisms $f \to g$ in $\TwArr(\mc C)$ are commutative diagrams
  \[
    \begin{tikzcd}
      X \ar[r,"f"] & Y \ar[d] \\
      Z \ar[r,"g",swap] \ar[u] & W.
    \end{tikzcd}
  \]
  given by precomposition and postcomposition with morphisms in $\mc C$. The source and target determine forgetful maps $\Ar(\mc C) \to \mc C \times \mc C$ and $\TwArr(\mc C) \to \mc C^{op} \times \mc C$.
\end{reminder}

\begin{proposition}\label{prop:diagrammaps}
  Given functors $\alpha, \beta\co K \to \mc C$, in the functor $\infty$-category $\Fun(K,\mc C)$ we have
  \[
    \Map_{\Fun(K, \mc C)} (\alpha, \beta) \simeq \holim_{f\co i \to j} \Map_{\mc C}(\alpha(i), \beta(j)).
  \]
  Here the homotopy limit is taken over $f$ in the twisted arrow category $\TwArr(K)$.
\end{proposition}

\begin{proof}
  This is shown by Glasman \cite[Lemma~2.3]{GlasmanTHH} (and, later, Gepner--Haugseng--Nikolaus \cite[Proposition~5.1]{GepnerHaugsengNikolausFibrations}), where it is equivalently described as an end
  \[
  \int^{i \in K} \Map_{\mc C}(\alpha(i), \beta(i)).\qedhere
  \]
\end{proof}

\begin{corollary}\label{cor:arrowmaps}
  Suppose that we have maps $f\co X_0 \to X_1$ and $g\co Y_0 \to Y_1$ in $\mc C$, viewed as objects in $\Ar(\mc C)$. Then there is a homotopy pullback diagram of function spaces:
  \[
    \begin{tikzcd}
      \Map_{\Ar(\mc C)}(f,g) \ar[r] \ar[d] &
      \Map_{\mc C}(X_0, Y_0) \ar[d] \\
      \Map_{\mc C}(X_1, Y_1) \ar[r] &
      \Map_{\mc C}(X_0, Y_1)
    \end{tikzcd}
  \]
\end{corollary}

\begin{proof}
  This follows by identifying the twisted arrow category of $\Delta^1 = \{0 \to 1\}$, which is an ordinary discrete category, with the 3-object poset
  \[
  \begin{tikzcd}
    & (0 \to 0) \ar[d] \\
    (1 \to 1) \ar[r] & (0 \to 1)
  \end{tikzcd}
  \]
  parametrizing cospans.
\end{proof}

\begin{lemma}\label{lem:sectionmaps}
  Suppose that $p\co \mc C \to K$ is a fibration of $\infty$-categories and $\alpha, \beta\co K\to \mc C$ are sections. Then there is an equivalence
  \[
    \Map_{\Fun_{K}(K,\mc C)} (\alpha, \beta) \simeq \holim_{f\co i \to j} \Map_{\mc C}(\alpha(i), \beta(j))_{f}.
  \]
  Here the limit is taken over $f\co i \to j$ in the twisted arrow category $\TwArr(K)$, and $\Map_{\mc C}(X, Y)_{f}$ denotes the homotopy fiber of $\Map_{\mc C}(X,Y) \to \Map_{K}(pX, pY)$ over $f$.
\end{lemma}

\begin{proof}
  Under the identification of Proposition~\ref{prop:diagrammaps}, the identity natural transformation $\id_{K} \to \id_{K}$ in the functor category $\Fun(K, K)$ is represented by the element
  \[
    \{\id_K\} \simeq \holim_{f\co i \to j} \{f\} \subset \holim_{f\co i \to j} \Map_{K}(i,j).
  \]
  The section category is a (homotopy) fiber product:
  \[
    \Fun_{K}(K, \mc C) = \Fun(K, \mc C) \timesover{\Fun(K, K)} \{\id_{K}\}
  \]
  Therefore, because mapping spaces in this pullback $\infty$-category are (homotopy) fiber products of function spaces and these commute with homotopy limits, the space $\Map_{\Fun_{K}(K, \mc C)}(\alpha, \beta)$ is the homotopy fiber product
  \[
    \holim_{f\co i \to j} \left(\Map_{\mc C}(\alpha(i), \beta(j)) \timesover{\Map_{K}(i,j)} \{f\}\right).
  \]
  By definition, this is the homotopy limit of $\Map_{\mc C}(\alpha(i),\beta(j))_{f}$.
\end{proof}

\begin{corollary}\label{cor:cocartesiansectionmaps}
  Suppose that $p\co \mc C \to K$ is a coCartesian fibration of $\infty$-categories, allowing any map $f\co i \to j$ in $K$ to be  lifted to a functor $f_!\co \mc C_i \to \mc C_j$ from the fiber over $i$ to the fiber over $j$. Then for any sections $\alpha, \beta\co K\to \mc C$ there is an equivalence
  \[
    \Map_{\Fun_{K}(K,\mc C)} (\alpha, \beta) \simeq \holim_{f\co i \to j} \Map_{\mc C_j}(f_! \alpha(i), \beta(j)).
  \]
\end{corollary}

\begin{proof}
  The functor $f_!$ is defined so that there is a map $\eta_x\co x \to f_!(x)$ over $f$ with a universal property: for any $y$ in $\mc C$, the diagram
  \[
    \begin{tikzcd}
      \Map_{\mc C}(f_! x, y) \ar[r,"(-)\circ \eta_x"] \ar[d] & \Map_{\mc C}(x,y) \ar[d] \\
      \Map_K(j,py) \ar[r,"(-)\circ f",swap] & \Map_K(i, py)
    \end{tikzcd}
  \]
  is a homotopy pullback \cite[Proposition~2.4.4.3]{LurieHTT}. In particular, for any $y \in \mc C_j$, taking the fiber over $\{\id_j\} \in \Map_{K}(j, j) = \Map_{K}(j,py)$ shows that the map
  \[
    \Map_{\mc C_j}(f_! x, y) \to \Map_{\mc C}(x,y)_f
  \]
  is an equivalence. We then substitute this into Lemma~\ref{lem:sectionmaps}.
\end{proof}

The next lemma is typically part of the equivalence between $\infty$-categories and complete Segal spaces; it is included for reference.

\begin{lemma}\label{lem:moduliequiv}
  A functor $f\co \mc C \to \mc D$ between $\infty$-categories is an equivalence if the induced maps $\mc C^\simeq \to \mc D^\simeq$ and $\Ar(\mc C)^\simeq \to \Ar(\mc D)^\simeq$ of maximal subgroupoids are weak equivalences of Kan complexes.
\end{lemma}

\begin{proof}
  We need to show that such a functor $f$ is fully faithful and essentially surjective in the sense of \cite[\S 1.2.10]{LurieHTT}: it induces homotopy equivalences $\Map_{\mc C}(X,Y) \to \Map_{\mc C}(fX, fY)$ and every object in $\mc D$ is equivalent to one in the image. The isomorphism $\pi_0(\mc C^\simeq) \to \pi_0(\mc D^\simeq)$ shows that $f$ induces a bijection on homotopy equivalence classes of objects: in particular, it is essentially surjective. To show full faithfulness, for any objects $X$ and $Y$ of $\mc C$ there is a natural homotopy pullback diagram
  \[
  \begin{tikzcd}
    \Map_{\mc C}(X,Y) \ar[r] \ar[d] &
    \Ar(\mc C)^\simeq \ar[d] \\
    \{X,Y\} \ar[r] &
    \mc C^\simeq \times \mc C^\simeq,
  \end{tikzcd}
  \]
  and similarly for $\mc D$. Therefore, $f$ induces homotopy equivalences $\Map_{\mc C}(X,Y) \to \Map_{\mc D}(fX, fY)$, which shows full faithfulness.
\end{proof}

\section{Centralizers and centers}\label{sec:centralizers}

In categories where there are no ``endomorphism'' objects, such as the category of groups, the \emph{center} of an object $M$ plays a very similar role. We will begin by recalling the definitions from \cite[\S 5.3.1]{LurieHA}.

\begin{reminder}
  A \emph{final object} in an $\infty$-category $\mc C$ is an object $Z$ such that, for any $X$, the mapping space $\Map_{\mc C}(X,Z)$ is contractible.
\end{reminder}

\begin{definition}[{\cite[Definition 5.3.1.2]{LurieHA}}]\label{def:centralizer}
  Suppose that $\mc C$ is monoidal and that $\mc M$ is left-tensored over $\mc C$. A \emph{centralizer} of a morphism $f\co M \to N$ in $\mc M$ is a final object in the $\infty$-category
  \[
    (\mc C_{\unit/}) \times_{(\mc M_{\unit \otimes M/})} (\mc M_{\unit \otimes M // N}).
  \]
  Here the functor $\mc C_{\unit/} \to \mc M_{\unit \otimes M/}$ is given by tensoring with $M$.
  
  More explicitly, it is a final object $\mf Z(f)$ in $\mc C$ in a category of $\E_0$-algebras $Z$ in $\mc C$ with a chosen (coherently) commuting diagram
  \[
    \begin{tikzcd}
      & Z \otimes M \ar[dr] \\
      \unit \otimes M \ar[ur, "\eta \otimes \id"] \ar[rr,"f",swap] && N.
    \end{tikzcd}
  \]
\end{definition}

Centralizers are functorial, in the following sense.

\begin{lemma}\label{lem:centralizerfunctor}
  Suppose that $\mc C$ is monoidal, that $\mc M$ is left-tensored over $\mc C$, and that morphisms $f$ in $\mc M$ have centralizers $\mf Z(f)$ in $\mc C$. Then there exists an essentially unique functor
  \[
    \mf Z\co \TwArr(\mc M) \to \Alg_{\E_0}(\mc C)
  \]
  extending this definition.
\end{lemma}

\begin{proof}
  The diagram of $\infty$-categories
  \[
    \mc C_{\unit /} \to \mc M_{\unit \otimes M / } \from \mc M_{\unit \otimes M / / N}
  \]
  is functorial in $f\co \unit \otimes M \to N \in \TwArr(\mc M)$, and therefore so is the fiber product category. Since centralizers exist, each of these categories has a final object $\mf Z(f)$. By \cite[Proposition 2.4.4.9]{LurieHTT} this can be made into a functorial assignment of final object, which we can compose with the projection to $\mc C_{\unit/}$ to get a functor $\mf Z\co \TwArr(\mc M) \to \mc C_{\unit/}$.
\end{proof}

\begin{reminder}
  If $\mc M$ is left-tensored over $\mc C$, there is a category $\LMod(\mc M)$ \cite[Definition~4.2.1.13]{LurieHA} that can be identified with an $\infty$-category of pairs $(A,M)$ of an algebra $A$ in $\mc C$ and a left $A$-module $M$ in $\mc M$. This decomposition corresponds to a pair of forgetful functors $U\co \LMod(\mc M) \to \mc M$ and $V: \LMod(\mc M) \to \Alg_{\E_1}(\mc C)$. Both of these are categorical fibrations \cite[\S 2.2.5]{LurieHTT}; this allows us to compute homotopical pullbacks as ordinary pullbacks.
\end{reminder}

\begin{definition}[{\cite[Definition 5.3.1.6]{LurieHA}}]\label{def:center}
  Suppose that $\mc C$ is monoidal and that $\mc M$ is left-tensored over $\mc C$. A \emph{center} of an object $M$ of $\mc M$ is a final object in the $\infty$-category
  \[
    \LMod(\mc M) \times_{\mc M} \{M\}.
  \]
  More explicitly, it is a final object $\mf Z(M)$ in $\mc C$ in a category of $\E_1$-algebras $A$ with a chosen $A$-module structure on $M$.
\end{definition}

\begin{remark}
  Unlike the centralizers in Lemma~\ref{lem:centralizerfunctor}, centers do not enjoy a very general notion of functoriality.
\end{remark}

Centers and centralizers are closely connected: a center $\mf Z(M)$ of an object $M$ is equivalent to a centralizer $\mf Z(\id_M)$ of the identity morphism of $M$ by \cite[Proposition~5.3.1.8]{LurieHA}. Our first goal will be to expand on these definitions in terms of mapping spaces.

\begin{proposition}\label{prop:mappingcentralizer}
  Suppose that we have an $\E_0$-algebra $Z$ in $\mc C$ and a diagram
  \[
    \begin{tikzcd}
      & Z \otimes M \ar[dr] \\
      \unit \otimes M \ar[ur, "\eta \otimes \id"] \ar[rr,"f",swap] && N
    \end{tikzcd}
  \]
  in $\mc M$. Then this diagram makes $Z$ into a centralizer $\mf Z(f)$ if and only if, for any $\E_0$-algebra $A$ in $\mc C$, the diagram
  \[
    \begin{tikzcd}
      \Map_{\E_0}(A,Z) \ar[r] \ar[d] &
      \Map_{\mc M}(A \otimes M, N) \ar[d] \\
      \{f\} \ar[r] &
      \Map_{\mc M}(\unit \otimes M, N)
    \end{tikzcd}
  \]
  is a homotopy pullback diagram.
\end{proposition}

\begin{proof}
  As stated, a centralizer of $f$ is a final object in the $\infty$-category
  \[
    \mc C_{\unit/} \times_{\mc M_{\unit \otimes M/}} \mc M_{\unit \otimes M / / N}.
  \]
  An object $\lambda_A$ in this consists of a map $\eta_A\co \unit \to A$, an action map $\lambda_A\co A \otimes M \to N$, and a homotopy making the diagram
  \[
    \begin{tikzcd}
      &A \otimes M \ar[dr,"\lambda_A"] \\
      \unit \otimes M \ar[ur,"\eta_A \otimes 1"] \ar[rr,"f",swap] && N
    \end{tikzcd}
  \]
  coherently commutative. Given two objects $(\eta_A, \lambda_A)$ and $(\eta_B, \lambda_B)$ in this category, we would like to compute the mapping space between them in this fiber product $\infty$-category. The mapping space between their images in $\mc C_{\unit/}$ is $\Map_{\E_0}(A,B)$, the mapping space between their images in $\mc M_{\unit \otimes M/}$ is the homotopy fiber over $\eta_B \otimes 1$ of 
  \[
    \Map_{\mc M}(A \otimes M, B \otimes M) \xrightarrow{(-)\circ \eta_A \otimes 1} \Map_{\mc M}(\unit \otimes M, B \otimes M),
  \]
  and the mapping space between their images in $\mc M_{\unit \otimes M / / N}$ is the homotopy fiber of
  \[
    \Map_{\mc M}(A \otimes M, B \otimes M) \to \Map_{\mc M}(A \otimes M, N) \times_{\Map_{\mc M}(\unit \otimes M, N)} \Map_{\mc M}(\unit \otimes M, B \otimes M)
  \]
  over $(\lambda_A, \eta_B \otimes 1)$. This allows us to describe the mapping spaces in this fiber product as computed by a homotopy pullback diagram:
  \[
    \begin{tikzcd}
      \Map((\eta_A, \lambda_A), (\eta_B, \lambda_B)) \ar[r] \ar[d] &
      \Map_{\E_0}(A,B) \ar[d] \\
      \{\lambda_A\} \ar[r] &
      \Map_{\mc M}(A \otimes M, N) \times_{\Map_{\mc M}(\unit \otimes M, N)} \{f\}
    \end{tikzcd}
  \]
  In order for $(\eta_Z, \lambda_Z)$ to be a final object in this category, the pullback space $\Map((\eta_A,\lambda_A), (\eta_Z,\lambda_Z))$ must always be contractible. This is equivalent to knowing that the map
  \[
    \Map_{\E_0}(A,Z) \to \Map_{\mc M}(A \otimes M, N) \times_{\Map_{\mc M}(\unit \otimes M, N)} \{f\}
  \]
  is a homotopy equivalence over the path component of any $\lambda_A$. However, all path components appear as some $\lambda_A$: a point of $\Map_{\mc M}(A \otimes M, N) \times_{\Map_{\mc M}(\unit \otimes M, N)} \{f\}$ is equivalent to a choice of $\lambda_A$ making $(\eta_A, \lambda_A)$ into an object of this category. Therefore, $Z$ is a centralizer if and only if the diagram
  \[
    \begin{tikzcd}
      \Map_{\E_0}(A,Z) \ar[r] \ar[d] &
      \Map_{\mc M}(A \otimes M, N) \ar[d] \\
      \{f\} \ar[r] &
      \Map_{\mc M}(\unit \otimes M, N)
    \end{tikzcd}
  \]
  is a homotopy pullback for all $\E_0$-algebras $A$.
\end{proof}

This characterization of centralizers in terms of mapping spaces allows us to determine the structure of centralizers in diagram categories. Recall that if $\mc M$ is left-tensored over $\mc C$, then $\Fun(K,\mc M)$ is left-tensored over $\Fun(K,\mc C)$ using the pointwise tensor product \cite[Remark 2.1.3.4]{LurieHA}. The diagonal functor $\mc C \to \Fun(K,\mc C)$ is compatible with this pointwise tensor, making $\Fun(K, \mc M)$ left-tensored over $\mc C$ as well, with an object-by-object definition:
\[
  (A \otimes F)(X) = A \otimes (F(X))
\]

\begin{theorem}\label{thm:diagramcentralizer}
  Suppose that $\mc M$ is left-tensored over $\mc C$ and that $\mc M$ has centralizers. For a natural transformation $\theta\co K \times \Delta^1 \to \mc M$ between diagrams $F, G\co K \to \mc M$, the homotopy limit
  \[
    \holim_{\gamma \in \TwArr(K)} \mf Z(\theta(\gamma, u))
  \]
  (if it exists) is a centralizer $\mf Z(\theta)$ in $\Fun(K,{\mc M})$. Here $u$ is the nonidentity morphism in $\Delta^1$. 
\end{theorem}

\begin{proof}
  For any fixed $\E_0$-algebra $A$ and $\gamma\co M \to N$ with image $\theta(\gamma,u)\co F(M) \to G(N)$, the diagram
  \[
    \begin{tikzcd}
      \Map_{\E_0}(A,\mf Z(\theta(\gamma))) \ar[r] \ar[d] &
      \Map_{\mc M}(A \otimes F(M), G(N)) \ar[d] \\
      \{\theta(\gamma,u)\} \ar[r] &
      \Map_{\mc M}(F(M),G(N))
    \end{tikzcd}
  \]
  is a homotopy pullback diagram in $\mc C$ by Proposition~\ref{prop:mappingcentralizer} and is functorial in the twisted arrow category $\TwArr(K)$. Taking homotopy limits over $\TwArr(K)$, Lemma~\ref{lem:sectionmaps} shows that we get a homotopy pullback diagram
  \[
    \begin{tikzcd}
      \holim_{\gamma \in \TwArr(K)}\Map_{\E_0}(A,\mf Z(\theta(\gamma,u))) \ar[r] \ar[d] &
      \Map_{\Fun(K,\mc M)}(A \otimes F, G) \ar[d] \\
      \{\theta\} \ar[r] &
      \Map_{\Fun(K,\mc M)}(F,G)
    \end{tikzcd}
  \]
  that is natural in $A$. If the homotopy limit $\holim_{\gamma \in \TwArr(K)} \mf Z(\theta(\gamma,u))$ exists in $\mc C$, the above homotopy pullback diagram and Proposition~\ref{prop:mappingcentralizer} then show that it is a centralizer of $\theta$ in the functor category.
\end{proof}

\begin{corollary}
  For a map $g\co M \to N$ in $\mc M$, the center $\mf Z^{\Ar(\mc M)}(g) \simeq \mf Z^{\Ar(\mc M)}(\id_g)$, where $\id_g$ is viewed as a map in the arrow category $\Ar(\mc M)$, is part of a homotopy pullback diagram
  \[
    \begin{tikzcd}
      \mf Z^{\Ar(\mc M)}(\id_g) \ar[r] \ar[d] & 
      \mf Z(M) \ar[d] \\
      \mf Z(N) \ar[r] &
      \mf Z(g).
    \end{tikzcd}
  \]
\end{corollary}

The next result makes explicit that the center serves as an endomorphism object among \emph{algebras} acting on $M$: not only is $\mf Z(M)$ final among algebras acting on $M$, but a map of algebras $A \to B$ over $\mf Z(M)$ is equivalent to a map of algebras compatible with action on $M$.

\begin{proposition}\label{prop:algebrasovercenter}
  For an object $M \in \mc M$ with center $\mf Z(M)$ in $\mc C$, there is an equivalence
  \[
    \Alg_{\E_1}(\mc C)_{/\mf Z(M)} \simeq 
    \LMod(\mc M) \times_{\mc M} \{M\}
  \]
  between the category of $\E_1$-algebras $A$ over $\mf Z(M)$ and the category of $\E_1$-algebras $A$ together with an $A$-module structure on $M$.
\end{proposition}

\begin{proof}
  As in Definition~\ref{def:center}, the center $\mf Z(M)$ is a final object in the category $\LMod_{\mc M} \times_{\mc M} \{M\}$.

  The forgetful map $\LMod(\mc M) \times_{\mc M} \{M\} \to \Alg_{\E_1}(\mc C)$ is a right fibration by \cite[Corollary~4.7.1.42]{LurieHA}, which means that for any  map $B \to C$ of algebras acting on $M$ there is a natural homotopy pullback diagram
  \[
    \begin{tikzcd}
      \Map_{\LMod(\mc M) \times_{\mc M} \{M\}}(A, B) \ar[r] \ar[d] &
      \Map_{\LMod(\mc M) \times_{\mc M} \{M\}}(A, C) \ar[d] \\
      \Map_{\Alg_{\E_1}(\mc C)} (A,B) \ar[r] &
      \Map_{\Alg_{\E_1}(\mc C)} (A,C)
    \end{tikzcd}
  \]
  by \cite[Proposition~2.4.4.3]{LurieHTT}. Taking $C = \mf Z(M)$, the upper-right space is contractible. Taking homotopy fibers of the horizontal maps, we find that the forgetful functor induces an equivalence
  \[
    \Map_{\LMod(\mc M) \times_{\mc M} \{M\}}(A, B) \to \Map_{\Alg_{\E_1}(\mc C)_{/\mf Z(M)}} (A,B)
  \]
  as desired.
\end{proof}
  
\section{Actions}\label{sec:actions}

In this section, we will assume that $\mc C$ is a monoidal $\infty$-category and that $\mc M$ is left-tensored over $\mc C$. To be more precise about this, we will begin with some background about how coherent left-tensorings are handled.

\begin{reminder}
  There is an ordinary category $LM^\otimes$ which is the universal example of a symmetric monoidal category with an algebra $\mf a$ and a left $\mf a$-module $\mf m$ \cite[Notation 4.2.1.6]{LurieHA}. The objects of $LM^\otimes$ are formal tuples of these objects, and the maps are generated by a unit $() \to (\mf a)$, an associative unital product $(\mf a, \mf a) \to (\mf a)$, and an associative unital left action $(\mf a, \mf m) \to (\mf m)$. The $\infty$-category $\mc{LM}^\otimes$ is the nerve of $LM^\otimes$.

  A monoidal category $\mc C$ with a category $\mc M$ left-tensored over it is encoded by a symmetric monoidal functor $\Psi$ from $\mc{LM}^\otimes$ to $\infty$-categories; this sends $\mf a$ to a monoidal $\infty$-category $\Psi(\mf a) = \mc C$ and sends $\mf m$ to an $\infty$-category $\Psi(\mf m) = \mc M$ with a left action of $\mc C$. This functor is equivalently encoded by a coCartesian fibration $\mc C^\otimes \to \mc{LM}^\otimes$, whose fiber over $X$ is $\Psi(X)$.
\end{reminder}

\begin{definition}\label{def:triple}
  Let $\sigma\co \Delta^2 \to \mc{LM}^\otimes$ represent the sequence of maps
  \[
    (\mf m) \to (\mf a, \mf m) \to (\mf m).
  \]
  Here the first map is induced by the unit of the algebra $\mf a$, and the second map represents the left action of $\mf a$ on $\mf m$.
\end{definition}

\begin{reminder}
  We can now form the fiber product $\mc B = \Delta^2 \times_{\mc{LM}^\otimes} \mc C^\otimes$. The functor $\mc B \to \Delta^2$ represents the restriction to a $\Delta^2$-shaped diagram of $\infty$-categories and functors
  \[
  \begin{tikzcd}
    &\mc C \times \mc M \ar[dr,"\otimes"]\\
    \mc M \ar[ur,"{(\unit,-)}"] \ar[rr,"\id"]&& 
    \mc M,
  \end{tikzcd}
  \]
  commuting up to natural isomorphism. We will now discuss categories of \emph{sections}: maps $\Delta^2 \to \mc C^\otimes$ over $\mc{LM}^\otimes$, or equivalently $\Delta^2 \to \mc B$ over $\Delta^2$. Sections make it possible to discuss a category whose objects are the following data:
  \begin{enumerate}
    \item objects $M_0, M_1, M_2$ in $\mc M$ and $Y \in \mc C$,
    \item maps $(\unit,M_0) \to (Y,M_1)$, $Y \otimes M_1 \to M_2$, and $M_0 \to M_2$,
    \item a coherence expressing commutativity of the resulting diagram
    \[
    \begin{tikzcd}
      \unit \otimes M_0 \ar[r] \ar[d,"\simeq"] &
      Y \otimes M_1 \ar[d] \\
      M_0 \ar[r] & M_2.
    \end{tikzcd}
    \]
  \end{enumerate}
\end{reminder}

\begin{proposition}\label{prop:sectionmapping}
  Given two sections
  \begin{align*}
    s_M &= (M_0 \to (X,M_1) \to M_2)\text{ and}\\
    s_N &= (N_0 \to (Y, N_1) \to N_2)
  \end{align*}
  in the functor category ${\Fun_{\mc{LM}^\otimes}}(\Delta^2, \mc C^\otimes)$, the mapping space between them is a homotopy pullback
  \[
    \begin{tikzcd}
      \Map_{\Fun_{\mc{LM}^\otimes}(\Delta^2, \mc C^\otimes)}(s_M, s_N) \ar[r] \ar[d]
      & \Map_{\mc M}(M_2, N_2) \ar[d] \\
      \Map_{\E_0}(X,Y) \times \Map_{\Ar(\mc M)}(M_0 \to M_1, N_0 \to N_1) \ar[r] &
      \Map_{\mc M}(X \otimes M_1, N_2).
    \end{tikzcd}
  \]

  Given two sections
  \begin{align*}
    t_M &= ((X,M_1) \from M_0 \to M_2)\text{ and}\\
    t_N &= ((Y,N_1) \from N_0 \to N_2)
  \end{align*}
  in the functor category $\Fun_{\mc{LM}^\otimes}(\Lambda^2_0,\mc C^\otimes)$, the mapping space between them is the homotopy pullback
  \[
    \begin{tikzcd}
      \Map_{\Fun_{\mc{LM}^\otimes}(\Lambda^2_0,\mc C^\otimes)}(t_M, t_N) \ar[r] \ar[d] &
      \Map_{\mc M}(M_2, N_2) \ar[d] \\
      \Map_{\E_0}(X,Y) \times \Map_{\Ar(\mc M)}(M_0 \to M_1, N_0 \to N_1) \ar[r] &
      \Map_{\mc M}(M_0, N_2).
    \end{tikzcd}
  \]
\end{proposition}

\begin{proof}
  By Corollary~\ref{cor:cocartesiansectionmaps}, the mapping space between the sections $s_M$ and $s_N$ in $\Map_{\Fun_{\mc{LM}^\otimes}}(\Delta^2, \mc C^\otimes)$ is the homotopy limit of the following commutative diagram, indexed by the twisted arrow category of $\Delta^2$:
  \[
    \begin{tikzcd}
      && \Map_{\mc M}(M_2, N_2) \ar[d] \ar[dd,bend left=90] \\
      & \Map_{\mc C}(X,Y) \times \Map_{\mc M}(M_1, N_1) \ar[r] \ar[d] &
      \Map_{\mc M}(X \otimes M_1, N_2) \ar[d] \\
      \Map_{\mc M}(M_0, N_0) \ar[r] \ar[rr,bend right=10] &
      \Map_{\mc C}(\unit, Y) \times \Map_{\mc M}(M_0, N_1) \ar[r] &
      \Map_{\mc M}(M_0, N_2)
    \end{tikzcd}
  \]
  By first taking pullback of the two maps contained in the left side of the diagram, we can re-express this homotopy limit as the desired pullback. The mapping space between sections $t_M$ and $t_N$ in $\Fun_{\mc{LM}^\otimes}(\Lambda^2_0,\mc C^\otimes)$ is similarly computed by the homotopy limit of the subdiagram
  \[
    \begin{tikzcd}
      & \Map_{\mc C}(X,Y) \times \Map_{\mc M}(M_1, N_1) \ar[d] &
      \\
      \Map_{\mc M}(M_0, N_0) \ar[r] \ar[rr,bend right=10] &
      \Map_{\mc C}(\unit, Y) \times \Map_{\mc M}(M_0, N_1) &
      \Map_{\mc M}(M_0, N_2),
    \end{tikzcd}
  \]
  again computed by first taking homotopy pullback of the left portion of the diagram.
  % Similarly, for two sections $t_M = ((X,M_1) \from M_0 \to M_2)$ and $t_N = ((Y,N_1) \from N_0 \to N_2)$ in the functor category $\Fun_{\mc{LM}^\otimes}(\Lambda^2_0,\mc C^\otimes)$, we get a pullback diagram
\end{proof}

\begin{corollary}\label{cor:actionmapping}
  For sections $s_M, s_N$ over $\Delta^2$ or $t_M, t_N$ over $\Lambda^2_0$ such that the maps $M_0 \to M_i$ and $N_0 \to N_i$ are all equivalences, we have a homotopy pullback diagram
  \[
    \begin{tikzcd}
      \Map(s_M, s_N) \ar[r] \ar[d] &
      \Map_{\mc M}(M_0,N_0) \ar[d] \\
      \Map{\E_0}(X,Y) \times \Map_{\mc M}(M_0,N_0) \ar[r] &
      \Map_{\mc M}(X \otimes M_0, N_0)
    \end{tikzcd}
  \]
  of function spaces in $\Fun_{\mc{LM}^\otimes}(\Delta^2, \mc C^\otimes)$, and a natural equivalence
  \[
    \begin{tikzcd}
      \Map(t_M, t_N) \to \Map_{\E_0}(X,Y) \times \Map_{\mc M}(M_0,N_0).
    \end{tikzcd}
  \]
  of function spaces in $\Fun_{\mc{LM}^\otimes}(\Lambda^2_0, \mc C^\otimes)$.
\end{corollary}

\begin{definition}\label{def:EZaction}
  For an $\E_0$-algebra $X$ in $\mc C$ and an object $M$ in $\mc M$, a \emph{left action} of $X$ is a map $\lambda_M\co X \otimes M \to M$ with a commutative diagram
  \[
    \begin{tikzcd}
      & X \otimes M \ar[dr, "\lambda_M"] \\
      \unit \otimes M \ar[ur, "\eta \otimes 1"] \ar[rr,"\simeq", swap] && M.
    \end{tikzcd}
  \]
\end{definition}

Our ultimate goal is to relate actions of $X$ to module structures over an algebra freely generated by $X$; this will occur when the category $\mc M$ has centralizers in $\mc C$. To prepare for this, the remainder of this section is devoted to  understanding categories of objects with action, and in particular spaces of maps between such objects.

\begin{definition}\label{def:EZmod}
  The categories $\EZMod_X(\mc M)$, of objects with left action by $X$, and $\EZMod(\mc M)$, of pairs of an $\E_0$-algebra and an object it acts on, are the $\infty$-categories defined by the following pullback squares:
  \[
    \begin{tikzcd}
      \EZMod_X(\mc M)\ar[r] \ar[d] &
      \EZMod(\mc M) \ar[r] \ar[d] &
      \Fun_{\mc{LM}^\otimes}(\Delta^2, \mc C^\otimes) \ar[d] \\
      \{X\} \times \mc M \ar[r] &
      \Alg_{\E_0}(M) \times \mc M \ar[r] &
      \Fun_{\mc{LM}^\otimes}(\Lambda^2_0,\mc C^\otimes).
    \end{tikzcd}
  \]
  Here the rightmost-bottom functor $\Alg_{\E_0}(M) \times \mc M \to \Fun_{\mc{LM}^\otimes}(\Lambda^2_0,\mc C)$ sends $(X,M)$ to the diagram
\[
  \begin{tikzcd}
    & X \otimes M \\
    \unit \otimes M \ar[ur, "\eta \otimes 1"] \ar[rr,"\simeq", swap] && M.
  \end{tikzcd}
\]
\end{definition}

\begin{proposition}\label{prop:EZmaps}
  Suppose that $\mc C$ is a monoidal $\infty$-category and that $\mc M$ is left-tensored over $\mc C$.

  The functor $\EZMod(\mc M) \to \Fun_{\mc{LM}^\otimes}(\Delta^2, \mc C^\otimes)$ is fully faithful, with essential image consisting of those sections $s_M  = (M_0 \to (X,M_1) \to M_2)$ such that $M_0 \to M_i$ are equivalences.
  
  If $X$ is an $\E_0$-algebra in $\mc C$ and $M$ and $N$ are objects in $\EZMod_X(\mc M)$, the natural diagram
  \[
  \begin{tikzcd}
    \Map_{\EZMod_X(\mc M)}(M,N) \ar[r] \ar[d] &
    \Map_{\mc M}(M,N) \ar[d] \\
    \Map_{\mc M}(M,N) \ar[r] &
    \Map_{\mc M}(X \otimes M, N)
  \end{tikzcd}
  \]
  is a homotopy pullback diagram. Here the top and left-hand maps are forgetful, while the bottom map is $f \mapsto \lambda_N \circ (1 \otimes f)$ and the right-hand map is $f \mapsto f \circ \lambda_M$.
\end{proposition}

\begin{proof}
  Corollary~\ref{cor:actionmapping} shows that the map $\Alg_{\E_0}(\mc C) \times \mc M \to \Fun_{\mc{LM}^\otimes}(\Lambda^2_0, \mc C^\otimes)$ is fully faithful, with essential image consisting of those sections such that $M_0 \to M_1$ and $M_0 \to M_2$ are equivalences. The pullback map $\EZMod(\mc M) \to \Fun_{\mc{LM}^\otimes}(\Delta^2, \mc C^\otimes)$ is therefore fully faithful, with mapping spaces computed as above.

  To compute mapping spaces in $\EZMod_X(\mc M)$ for $X$ a fixed $\E_0$-algebra, we have to take the fiber product over $\Alg_{\E_0}(\mc C)$ with $\{X\}$. For objects $M$ and $N$ with action by $X$, we form the homotopy fiber product
  \[
    \Map_{\EZMod(\mc M)}((X,M),(X,N)) \times_{\Map_{\Alg_{\E_0}(\mc C)}(X,X)} \{\id_X\}
  \]
  which is the desired pullback diagram by Corollary~\ref{cor:actionmapping}.
\end{proof}

\begin{proposition}\label{prop:EZpushout}
  Suppose that the forgetful functor $U\co \EZMod_X(\mc M) \to \mc M$ has a left adjoint $L$. Then there exists a natural transformation $\lambda\co L(X \otimes P) \to L(P)$, such that for any object $M \in \EZMod_X(\mc M)$ with left action $\lambda_M\co X\otimes UM \to UM$, the diagram
  \[
    \begin{tikzcd}
      L(X \otimes UM) \ar[r,"\lambda"] \ar[d,"L(\lambda_M)",swap] &
      LUM \ar[d,"\epsilon"] \\
      LUM \ar[r,"\epsilon",swap] &
      UM
    \end{tikzcd}
  \]
  is a homotopy pushout in $\EZMod_X(\mc M)$.
\end{proposition}

\begin{proof}
  The action of $X$ on $L(P)$ determines a composite
  \[
  X \otimes P \to X \otimes L(P) \to L(P)
  \]
  in $\mc M$, whose adjoint is a natural map $\lambda\co L(X \otimes P) \to L(P)$. This has the property that for any map $M \to N$ in $\EZMod_X(\mc M)$, the composite
  \[
    L(X \otimes UM) \xrightarrow{\lambda} L(UM) \to N
  \]
  is adjoint to the map $X \otimes UM \to X \otimes UN \xrightarrow{\lambda_N} UN$.
  
  The adjunction means that, by Proposition~\ref{prop:EZmaps}, there is a natural homotopy pullback diagram
  \[
    \begin{tikzcd}
      \Map_{\EZMod_X}(M,N) \ar[r,"\epsilon^*"] \ar[d,"\epsilon^*",swap] &
      \Map_{\EZMod_X}(LUM,N) \ar[d,"\lambda^*"] \\
      \Map_{\EZMod_X}(LUM,N) \ar[r,"L(\lambda_M)^*",swap] &
      \Map_{\EZMod_X}(L(X \otimes UM),N).
    \end{tikzcd}
  \]
  Since $N$ is an arbitrary object in $\EZMod_X$, this implies that there is a corresponding homotopy pushout diagram for $M$.
\end{proof}

\begin{proposition}\label{prop:actionsovercenter}
  For an object $M \in \mc M$ with center $\mf Z(M)$, there is an equivalence
  \[
    \Alg_{\E_0}(\mc C)_{/\mf Z(M)} \simeq 
    \EZMod(\mc M) \times_{\mc M} \{M\}
  \]
  between the category of $\E_0$-algebras over $\mf Z(M)$ and the category of $\E_0$-algebras acting on $M$.
\end{proposition}

\begin{proof}
  For algebras $A$ and $B$ acting on $M$, there is a homotopy pullback diagram
  \[
    \begin{tikzcd}
      \Map_{\EZMod(\mc M)}((A,M), (B,M)) \ar[r] \ar[d] &
      \Map_{\mc M}(M,M) \ar[d] \\
      \Map_{\Alg_{\E_0}(\mc C)}(A,B) \times \Map_{\mc M}(M,M) \ar[r] &
      \Map_{\mc M}(A \otimes M,M)
    \end{tikzcd}
  \]
  by Corollary~\ref{cor:actionmapping}. Taking the fiber product over $\Map_{\mc M}(\unit \otimes M,M)$ with the canonical equivalence $\unit \otimes M \to M$, we get a homotopy pullback
  \[
    \begin{tikzcd}
      \Map_{\EZMod(\mc M) \times_{\mc M} \{M\}}(A, B) \ar[r] \ar[d] &
      \{\id_M\} \ar[d] \\
      \Map_{\Alg_{\E_0}(\mc C)}(A,B) \ar[r] &
      \Map_{\mc M_{\unit \otimes M /}}(A \otimes M,M).
    \end{tikzcd}
  \]
  This is natural in $B$, and the right-hand vertical map is independent of $B$. Naturality of this diagram in $B$ then shows that, for any map $B \to C$ in $\EZMod(\mc M) \times{\mc M} \{M\}$, we get a homotopy pullback diagram
  \[
    \begin{tikzcd}
      \Map_{\EZMod(\mc M) \times_{\mc M} \{M\}}(A, B) \ar[r] \ar[d] &
      \Map_{\EZMod(\mc M) \times_{\mc M} \{M\}}(A, C) \ar[d] \\
      \Map_{\Alg_{\E_0}(\mc C)}(A,B) \ar[r] &
      \Map_{\Alg_{\E_0}(\mc C)}(A,C).
    \end{tikzcd}
  \]
  Taking $C=\mf Z(M)$, a final object in the category of $\E_0$-algebras with a unital action on $M$ by definition, makes the upper-right corner contractible. Therefore, we find that the forgetful functor induces an equivalence
  \[
    \Map_{\EZMod(\mc M) \times_{\mc M} \{M\}}(A, B) \to \Map_{\Alg_{\E_0}(\mc C)_{/\mf Z(M)}} (A,B)
  \]
  as desired.
\end{proof}

\begin{remark}
  The above proof shows that, up to equivalence, the forgetful functor $\EZMod(\mc M) \times_{\mc M} \{M\} \to \Alg_{\E_0}(M)$ is a right fibration, which is classified by $\mf Z(M)$ when it exists.
\end{remark}
  
\section{Actions by free algebras and centers}\label{sec:centers}

\begin{definition}\label{def:envelopingalgebra}
  Suppose that $X$ is an $\E_0$-algebra in a monoidal $\infty$-category $\mc C$. An \emph{enveloping algebra} for $X$ is an $\E_1$-algebra $\E X$ with a map $X \to \E X$ of $\E_0$-algebras such that, for any $\E_1$-algebra in $\mc C$, the composite
  \[
    \Map_{\Alg_{\E_1}(\mc C)}(\E X, A) \to \Map_{\Alg_{\E_0}(\mc C)}(\E X, A) \to    \Map_{\Alg_{\E_0}(\mc C)}(X, A)
  \]
  is an equivalence.
\end{definition}

\begin{proposition}\label{prop:modactcat}
  Suppose that $\mc C$ is a monoidal $\infty$-category and that $\mc M$ is left-tensored over $\mc C$. Let $A$ be an $\E_1$-algebra in $\mc C$ and $f\co X \to A$ a map of $\E_0$-algebras. Consider the commutative diagram
  \[
    \begin{tikzcd}
      \LMod_A(\mc M) \ar[r] \ar[dr] & \EZMod_A(\mc M) \ar[r] \ar[d] & \EZMod_X(\mc M) \ar[dl] \\
      &\mc M.
    \end{tikzcd}
  \]
  If $\mc M$ has centralizers, then for any $M \in \mc M$, the induced map on fibers over $\{M\}$ is equivalent to the map of spaces
  \[
    \Map_{\Alg_{\E_1}(\mc C)}(A, \mf Z(M)) \to \Map_{\Alg_{\E_0}(\mc C)}(X, \mf Z(M))
  \]
  induced by forgetting to $\E_0$-algebras and precomposing with $f$.
\end{proposition}

\begin{proof}
  The horizontal functors in this diagram are the composites
  \[
    \LMod(\mc M) \times_{\Alg_{\E_1}(\mc C)} \{A\} \to 
    \EZMod(\mc M) \times_{\Alg_{\E_0}(\mc C)} \{A\} \to 
    \EZMod(\mc M) \times_{\Alg_{\E_0}(\mc C)} \{X\}.
  \]
  Taking fiber products with $\{M\}$ over $\mc M$, by Propositions~\ref{prop:algebrasovercenter} and \ref{prop:actionsovercenter} and we get the diagram
  \[
    \Alg_{\E_1}(\mc C)_{/\mf Z(M)} \times_{\Alg_{\E_1}(\mc C)} \{A\} \to
    \Alg_{\E_0}(\mc C)_{/\mf Z(M)} \times_{\Alg_{\E_0}(\mc C)} \{A\} \to
    \Alg_{\E_0}(\mc C)_{/\mf Z(M)} \times_{\Alg_{\E_0}(\mc C)} \{X\}.
  \]
  However, fibers of overcategories are mapping spaces, and so this is equivalent to the sequence of maps
  \[
    \Map_{\Alg_{\E_1}(\mc C)} (A, \mf Z(M)) \to
    \Map_{\Alg_{\E_0}(\mc C)} (A, \mf Z(M)) \to
    \Map_{\Alg_{\E_0}(\mc C)} (X, \mf Z(M)),
  \]
  as desired.
\end{proof}

\begin{theorem}\label{thm:EZequivalence}
  Suppose that $\mc C$ is a monoidal $\infty$-category with pullbacks, that $\mc M$ is left-tensored over $\mc C$, and that $\mc M$ has centralizers in $\mc C$.
  
  If $X$ is an $\E_0$-algebra in $\mc C$ and $\E X$ is an enveloping algebra for $X$, then the forgetful functor
  \[
    \LMod_{\E X}(\mc M) \to \EZMod_{X}(\mc M)
  \]
  is an equivalence of $\infty$-categories.
\end{theorem}

\begin{proof}
  By Proposition~\ref{prop:modactcat}, there is a forgetful map $\LMod_{\E X}(\mc M) \to \EZMod_X(\mc M)$ whose fiber over any $\{M\}$ is the equivalence
  \[
    \Map_{\Alg_{\E_1}(\mc C)}(\E X, \mf Z(M)) \to 
    \Map_{\Alg_{\E_0}(\mc C)}(X, \mf Z(M))
  \]
  of spaces. The map on maximal subgroupoids $\LMod_{\E X}(\mc M)^\simeq \to \EZMod_X(\mc M)^\simeq$ is then a map of spaces over $\mc M^\simeq$ that is an equivalence on the fiber over any object $\{M\}$, and therefore a weak equivalence.

  Applying the same result to the arrow category $\Ar(\mc M) = \Fun(\Delta^1, \mc M)$, which has centralizers in $\mc C$ by Theorem~\ref{thm:diagramcentralizer}, we find that there is an equivalence
  \[
    \Ar(\LMod_{\E X}(\mc M))^\simeq \to \Ar(\EZMod_X(\mc M))^\simeq.
  \]
  By Lemma~\ref{lem:moduliequiv}, the map $\LMod_{\E X}(\mc M) \to \EZMod_X(\mc M)$ is then an equivalence.
\end{proof}

\begin{theorem}\label{thm:mainpushout}
  Suppose that $\mc C$ is a monoidal $\infty$-category with pullbacks, that $\mc M$ is left-tensored over $\mc C$, and that $\mc M$ has centralizers in $\mc C$.
  
  Suppose $X$ is an $\E_0$-algebra in $\mc C$ and $\E X$ is an enveloping algebra for $X$. For $M$ any left $\E X$-module, the induced natural diagram
  \[
  \begin{tikzcd}
    \E X \otimes X \otimes M \ar[r,"\mu \otimes 1"] \ar[d,swap,"1 \otimes \lambda_M"] &
    \E X \otimes M \ar[d, "m"] \\
    \E X \otimes M \ar[r, "m", swap] &
    M
  \end{tikzcd}
  \]
  is a homotopy pushout diagram in $\LMod_{\E X}(\mc M)$.
\end{theorem}

\begin{proof}
  There is a natural equivalence
  \[
    \Map_{\LMod_{\E X}(\mc M)}(\E X \otimes Y, N) \simeq \Map_{\mc M}(Y, N)
  \]
  by \cite[Corollary 4.2.4.8]{LurieHA}: the functor $\E X \otimes (-)$ is left adjoint to the forgetful functor. The result then follows by Proposition~\ref{prop:EZpushout}.
\end{proof}

\section{Algebras and operations}\label{sec:Ek}

In this section, we will discuss how the previous results can be applied in categories of $\mc O$-algebras for $\mc O$ an operad. In particular, we will exploit the Dunn additivity theorem in the form for $\infty$-categories proved in \cite[Theorem 5.1.2.2]{LurieHA}, which implies that
\[
  \Alg_{\E_{n+m}}(\mc C) = \Alg_{\E_n} (\Alg_{\E_m}(\mc C)).
\]

\begin{reminder}
  The Boardman--Vogt tensor product of operads \cite{BoardmanVogtRecognition} has a version for $\infty$-operads. Given two $\infty$-operads $\mc P$ and $\mc Q$, there is a tensor product $\mc P \otimes \mc Q$ \cite[Proposition~2.2.5.6]{LurieHA}, whose defining property is that $\mc P \otimes \mc Q$-algebras are the same as $\mc P$-algebras in the category of $\mc Q$-algebras:
  \[
    \Alg_{\mc P}(\Alg_{\mc Q}(\mc C)) \simeq \Alg_{\mc P \otimes \mc Q}(\mc C)
  \]
  In the following we will be considering the case where $\mc P$ is an $\E_1$-operad.
\end{reminder}

\begin{proposition}\label{prop:Opushout}
  Suppose that $\mc O$ is a coherent $\infty$-operad, $\mc C$ is a presentable $\E_1 \otimes \mc O$-monoidal $\infty$-category, $B$ is an $\E_1 \otimes \mc O$-algebra in $\mc C$, $A \to B$ is a map of $\E_0 \otimes \mc O$-algebras in $\mc C$, and that $\E A$ is an enveloping $\E_1 \otimes \mc O$-algebra for $A$. Given $R$ any $\mc O$-algebra in $\LMod_B$, the induced natural diagram
  \[
  \begin{tikzcd}
    B \otimes A \otimes R \ar[r,"m \otimes 1"] \ar[d,swap,"1 \otimes m"] &
    B \otimes R \ar[d, "m"] \\
    B \otimes R \ar[r, "m", swap] &
    B \otimes_{\E A} R
  \end{tikzcd}
  \]
  is a homotopy pushout diagram in the category of $\mc O$-algebras in $\LMod_{B}$.
\end{proposition}

\begin{proof}
  Let $\mc C$ and $\mc M$ be the category $\Alg_{\mc O}(\mc C)$, which is monoidal \cite[Example 3.2.4.4]{LurieHA} and has centralizers \cite[Theorem 5.3.1.14]{LurieHA}. In this category, $A$ is an $\E_0$-algebra, $\E A$ is an enveloping $\E_1$-algebra for $A$, and $R$ is a left $\E A$-module. The hypotheses of Theorem~\ref{thm:mainpushout} apply, and we get a homotopy pushout diagram
  \[
  \begin{tikzcd}
    \E A \otimes A \otimes R \ar[r,"m \otimes 1"] \ar[d,swap,"1 \otimes m"] &
    \E A \otimes R \ar[d, "m"] \\
    \E A \otimes R \ar[r, "m", swap] &
    R
  \end{tikzcd}
  \]
  of $\mc O$-algebras in $\LMod_{\E A}$.
  
  The forgetful functor $\Alg_{\mc O}(\LMod_B) \to \Alg_{\mc O}(\LMod_{\E A})$ preserves and detects homotopy limits because the composite forgetful functors to $\mc C$ do, so the left adjoint $B \otimes_{\E A} (-)$ preserves homotopy colimits. Applying this to the above homotopy pushout diagram gives the desired result.
\end{proof}

Specializing to the case of an $\E_k$-operad, the Dunn additivity theorem implies the following.
\begin{theorem}\label{thm:Ekpushout}
  Suppose that $\mc C$ is a presentable $\E_{k+1}$-monoidal $\infty$-category, that $B$ is an $\E_{k+1}$-algebra in $\mc C$, $A \to B$ is a map of $\E_k$-algebras in $\mc C$, and that $\E A$ is an enveloping $\E_{k+1}$-algebra for $A$. Given $R$ any $\E_k$-algebra in $\LMod_B$, the induced natural diagram
  \[
  \begin{tikzcd}
    B \otimes A \otimes R \ar[r,"m \otimes 1"] \ar[d,swap,"1 \otimes m"] &
    B \otimes R \ar[d, "m"] \\
    B \otimes R \ar[r, "m", swap] &
    B \otimes_{\E A} R
  \end{tikzcd}
  \]
  is a homotopy pushout diagram in the category of $\E_k$-algebras in $\LMod_{R}$.
\end{theorem}

\begin{remark}
  When $R = \unit$ is the monoidal unit and $\mc C$ is symmetric monoidal presentable, this result is an immediate consequence of \cite[{Proposition 5.2.2.12, Corollary 5.3.1.16}]{LurieHA}.
\end{remark}

\section{Special cases}\label{sec:cases}

\subsection{Associative algebras}

Our first examples are when $k=0$.

\begin{proposition}
  Suppose that $\mc C$ is stable presentable symmetric monoidal, and let $\T Y$ be the free associative algebra on $Y$. Then for any left $\T Y$-module $M$, there is a cofiber sequence
  \[
    \T Y \otimes Y \otimes M \to \T Y \otimes M \to M
  \]
  of left $\T Y$-modules.
\end{proposition}

For $V$ a flat module over a commutative ring $k$ and $N$ a left module over the tensor algebra $\T(V)$ whose underlying $k$-module is flat, this generalizes the standard Koszul resolution
\[
  0 \to \T(V) \otimes_k V \otimes_k N \to \T(V) \otimes_k N \to N \to 0.
\]

\begin{proof}
  Let $X$ be $\unit \oplus Y$, the free $\E_0$-algebra on $Y$, so that $\E X \simeq \T Y$. For any left $\T Y$-module $M$, Theorem~\ref{thm:Ekpushout} implies that there is a homotopy pushout diagram
  \[
    \begin{tikzcd}
      \T Y \otimes (\unit \oplus Y) \otimes M \ar[r] \ar[d] &
      \T Y \otimes M \ar[d] \\
      \T Y \otimes M \ar[r] &
      M
    \end{tikzcd}
  \]
  of left $\T Y$-modules. By observing that the upper-left decomposes into a sum $(\T Y \otimes M) \oplus (\T Y \otimes Y \otimes M)$, with first factor mapping by an equivalence to both corners, we find that we can re-express as the desired cofiber sequence.
\end{proof}

\subsection{Commutative algebras}

We note that when $k = \infty$, Theorem~\ref{thm:Ekpushout} recovers a known description of  homotopy pushouts of $\E_\infty$ rings.

\begin{proposition}
  Suppose that $\mc C$ is symmetric monoidal presentable. For a diagram $B \from A \to R$ of $\E_\infty$-algebras in $\mc C$, the homotopy pushout is equivalent to $B \otimes_A R$.
\end{proposition}

\begin{proof}
  When $k=\infty$, the enveloping algebra $\E A$ is always equivalent to $A$, $\E_\infty$-algebras in left $B$-modules are the same as $\E_\infty$-algebras with a map from $B$, and the coproduct is the tensor product. Therefore, this reduces us to observing that the homotopy pushout of $B \from A \to R$ is always equivalent to the homotopy pushout of $B \amalg R \from B \amalg A \amalg R \to B \amalg R$.
\end{proof}

\subsection{Augmented cases}

The most straightforward general situation, which eliminates worries about left module structures, is when $B$ is the monoidal unit $\unit$.

\begin{corollary}
  Suppose that $\mc C$ is a presentable $\E_{k+1}$-monoidal $\infty$-category with unit $\unit$, $\epsilon\co A \to \unit$ is an augmented $\E_k$-algebra in $\mc C$, and that $\E A$ is an enveloping $\E_{k+1}$-algebra for $A$. Given $R$ any $\E_k$-algebra in $\mc C$, the induced natural diagram
  \[
  \begin{tikzcd}
    A \otimes R \ar[r,"\epsilon \otimes 1"] \ar[d,swap,"m"] &
    R \ar[d] \\
    R \ar[r] &
    \unit \otimes_{\E A} R
  \end{tikzcd}
  \]
  is a homotopy pushout diagram in the category of $\E_k$-algebras in $\mc C$.
\end{corollary}

When $R$ is also the monoidal unit, this recovers the following description relating ``suspension'' to bar constructions on enveloping algebras.

\begin{corollary}\label{cor:barsuspension}
  Suppose that $\mc C$ is a presentable $\E_{k+1}$-monoidal $\infty$-category with unit $\unit$, $\epsilon\co A \to \unit$ is an augmented $\E_k$-algebra in $\mc C$, and that $\E A$ is an enveloping $\E_{k+1}$-algebra for $A$. Then the diagram
  \[
  \begin{tikzcd}
    A \ar[r,"\epsilon \otimes 1"] \ar[d,swap,"m"] &
    \unit \ar[d] \\
    \unit \ar[r] &
    \unit \otimes_{\E A} \unit
  \end{tikzcd}
  \]
  is a homotopy pushout diagram in the category of $\E_k$-algebras in $\mc C$.
\end{corollary}

Specializing again to the case where $A$ is a free algebra, we get the following result.
\begin{corollary}
  Suppose that $W$ is augmented over the unit $\unit$. Then there exists an equivalence
  \[
    \Free_{\E_k}(\hocolim(\unit \from W \to \unit)) \simeq \twobar(\unit, \Free_{\E_{k+1}}(W), \unit).
  \]
\end{corollary}

\begin{proof}
  The universal property of a free algebra identifies $\Free_{\E_{k+1}}(W)$ with the enveloping algebra $\E(\Free_{\E_k}(W))$.
\end{proof}

A result like this plays an important role in Basterra and Mandell's determination of homology for $\E_n$-algebras via an iterated bar construction \cite[Lemma~7.5]{BasterraMandellEn}, and a generalization of this appears as \cite[Corollary~5.2.2.13]{LurieHA}.

Another situation of particular interest to us, due to its connection to cell attachment, is when $R$ is the monoidal unit $\unit$ and the algebra $A$ is a free algebra.

\begin{proposition}\label{prop:freeEkpushout}
  Suppose that $\mc C$ is a presentable $\E_{k+1}$ monoidal $\infty$-category, that $B$ is an $\E_{k+1}$-algebra in $\mc C$, and that there is a diagram $\unit \xleftarrow{\epsilon} W \to B$ in $\mc C$. Then there is a homotopy pushout diagram
  \[
  \begin{tikzcd}
    B \otimes \Free_{\E_k}(W) \ar[r, "m"] \ar[d,"1 \otimes \bar \epsilon", swap] &
    B \ar[d] \\
    B \ar[r] &
    B \otimes_{\Free_{\E_{k+1}}(W)} \unit
  \end{tikzcd}
  \]
  of $\E_k$-algebras in $\LMod_B$.
\end{proposition}

The following result can be interpreted as showing that, if $B$ is \emph{central} in $R$, any elements in $B$ that become trivial in $R$ also have their $\E_{k+1}$-operations become trivial, even if those operations are not defined on $R$.
\begin{proposition}\label{prop:EKkilling}
  Suppose that $\mc C$ is a presentable $\E_{k+1}$-monoidal $\infty$-category, $W$ is an object of $\mc C$, $R$ is an $\E_k$-algebra in $\mc C$, and that there is a map of $\E_k$-algebras $B \to R$ that lifts to a map of $\E_{k+1}$-algebras from $B$ to the center $\mf Z(R)$. If there is a commutative diagram
  \[
  \begin{tikzcd}
    W \ar[r,"f"] \ar[d,"\epsilon", swap] & B \ar[d] \\
    \unit \ar[r] & R,
  \end{tikzcd}
  \]
  in $\mc C$, then it extends to a commutative diagram
  \[
  \begin{tikzcd}
    \Free_{\E_{k+1}}(W) \ar[r,"\tilde f"] \ar[d,swap,"\tilde \epsilon"] &
    B \ar[d] \\
    \unit \ar[r] & R.
  \end{tikzcd}
  \]
\end{proposition}

\begin{proof}
  Using the adjunction between $\mc C$ and $\E_k$-algebras we can extend the original square from $W$ to $\Free_{\E_k}(W)$, and then using the adjunction between $\E_k$-algebras in $\mc C$ and in $\LMod_B$, we get a commutative diagram
  \[
  \begin{tikzcd}
    B \otimes \Free_{\E_k}(W) \ar[r, "m"] \ar[d, "1 \otimes \tilde\epsilon", swap] &
    B \ar[d] \\
    B \ar[r] &
    R
  \end{tikzcd}
  \]
  of $\E_k$-algebras in $\LMod_B$; the unit map $B \to R$ therefore factors through the homotopy pushout. By the previous proposition, this is a factorization
  \[
  B \to B \otimes_{\Free_{\E_{k+1}}(W)} \unit \to R,
  \]
  which proves that the map $\Free_{\E_{k+1}}(W) \to R$ factors through $\tilde \epsilon$.
\end{proof}

\subsection{Classifying spaces}

A connection to the theory of classifying spaces appears when $\mc C = \mc M$ is the category of spaces, with Cartesian product.

\begin{example}
An $\E_0$-algebra in spaces is a pointed space, and for a well-pointed space $X$ the free $\E_1$-algebra on $X$ is modeled by the free associative algebra: the James construction $J(X)$.

Taking $B = R = \ast$, our main result then implies that there is a homotopy pushout diagram
\[
  \begin{tikzcd}
    X \ar[r] \ar[d] &
    \ast \ar[d] \\
    \ast \ar[r] &
    \twobar(\ast,{J(X)},\ast)
  \end{tikzcd}
\]
of spaces. In other words, there is an equivalence
\[
  \Sigma X \simeq B(J(X))
\]
between the classifying space of the James construction and the suspension of $X$. Moreover, this arises from a well-known homotopy pushout diagram
\[
  \begin{tikzcd}
    J(X) \times X \ar[r] \ar[d] &
    J(X) \ar[d] \\
    J(X) \ar[r] &
    \ast
  \end{tikzcd}
\]
of left modules over $J(X)$. This extends to a general description of $\twobar(\ast, J(X), M)$ as a homotopy pushout for any space $M$ with a left action of $J(X)$.
\end{example}

\begin{example}
  In the category of spaces, the easiest examples of $\E_k$-algebras are $k$-fold loop spaces. Because any map of $\E_k$-algebras $\Omega^k Y \to B$ must factor through the subspace $B^\times$ of grouplike elements, the recognition principle for $k$-fold and $(k+1)$-fold loop spaces \cite{MayLoopspaces} implies that the enveloping algebra $\E(\Omega^k Y)$ is $\Omega^{k+1} \Sigma Y$ (assuming $Y$ is $(k-1)$-connected).

  In these cases, the pushout diagrams in question are not new: they are gotten by applying $\Omega^k$ to homotopy pushout diagrams of $(k-1)$-connected spaces. Any new utility is that these pushouts still hold in \emph{non}-grouplike situations.
\end{example}

\section{Applications}\label{sec:applications}

We now specialize to applications in stable homotopy theory.

\subsection{$\E_3$ ring spectra}

% \begin{example}
%   Suppose that $B$ is a commutative ring spectrum. Then letting $A$ be the free $\E_k$ $B$-algebra on $B \vee W$ for any left $B$-module $W$, we find that there exists a homotopy pushout diagram
%   \[
%     \begin{tikzcd}
%       \Free_{\E_k}^B(W) \ar[r] \ar[d] & B \ar[d] \\
%       B \ar[r] & B \wedge^{\mb L}_{\Free_{\E_{k+1}}^B(W)} B.
%     \end{tikzcd}
%   \]
%   of $\E_k$-algebras in $\LMod_B$.
% \end{example}

\begin{theorem}\label{thm:e3smash}
  Fix a prime $p$, and consider the pair of $\E_3$-algebra maps $\Free_{\E_3}(S^0) \rightrightarrows \mb S$, classifying $0$ and $p$ respectively. Then the relative smash product is equivalent to the Eilenberg--Mac Lane spectrum $H\F_p$:
  \[
    H\F_p \simeq \mb S \smashover{\Free_{\E_3}(S^0)} \mb S
  \]
\end{theorem}

\begin{proof}
  A theorem of Mahowald at $p=2$, and of Hopkins at odd primes, is that the Eilenberg--Mac Lane spectrum $H\F_p$ is the Thom spectrum of a stable ($p$-local) spherical fibration over $\Omega^2 S^3$. This gives an alternative description of $H\F_p$ as the universal example $\mb S \aquot_{\E_2} p$ of an $\E_2$-algebra over $\mb S$ with a nullhomotopy of the element $p \in \pi_0(\mb S)$. This makes it the homotopy pushout of a diagram
  \[
    \begin{tikzcd}
      \Free_{\E_2}(S^0) \ar[r,"p"] \ar[d,"0",swap] &
      \mb S \ar[d] \\
      \mb S \ar[r] &
      H\mb F_p
    \end{tikzcd}
  \]
  of $\E_2$-algebras, and hence a relative tensor over the free $\E_3$-algebra by Corollary~\ref{cor:barsuspension}.
\end{proof}

\begin{reminder}
  For the purposes of this paper, a \emph{tower} of spectra is a functor from the poset $\N^{op}$ to spectra,
  \[
  \dots \to X_2 \to X_1 \to X_0
  \]
  whose \emph{underlying object} is $X_0$ and whose \emph{associated graded} is the graded spectrum $\{X_n / X_{n+1}\}_{n \in \N}$. We will write $F_k X$ for the tower
  \[
  \dots \to \ast \to X \to \dots \to X
  \]
  which is free on $X$ in filtration $k$; there is a natural map $F_k X  \to F_{k+1} X$.
  
  The monoidal structure on $\N$ gives towers a symmetric monoidal $\infty$-category structure under the Day convolution, where
  \[
  (X_\bullet \wedge Y_\bullet)_n = \hocolim_{p+q\geq n} X_p \wedge Y_q,
  \]
  and $F_0 S^0$ is the monoidal unit $\SS$. A \emph{filtered $\E_k$-algebra} is an $\E_k$-algebra in the category of towers of spectra. Both the underlying and associated graded functors are strong symmetric monoidal, and so preserve $\E_k$-algebras and (relative) smash products.
\end{reminder}

\begin{theorem}\label{thm:weirdmayfilt}
  For any prime $p$, there exists a filtered $\E_2$-algebra $R$ whose underlying spectrum is $\SS$ and whose associated graded is $H\F_p \wedge \Free_{\E_3}(S^0)$, with $S^0$ in filtration $1$ representing the image of $p$.
\end{theorem}

\begin{proof}
  We will write $\mb P^{\E_3}(x_n)$ for the free filtered $\E_3$-algebra $\Free_{\E_3}(F_n S^0)$ on a generator $x_n$ in filtration $n$, and $\mb P_{gr}^{\E_3}(y_n)$ for the free graded $\E_3$-algebra on a generator $y_n$ in grading $n$.
  
  There is a diagram of filtered $\E_3$-algebras
  \[
  \SS \from \mb P^{\E_3}(x_0) \to \mb P^{\E_3}(y_1),
  \]
  where the left-hand map sends $x_0$ to $p \in \pi_0 \SS$ and the right-hand map sends $x_0$ to the image of $y_1$ in filtration zero. We define $R$ to be the (derived) smash product
  \[
  R = \SS \smashover{\mb P^{\E_3}(x_0)} \mb P^{E_3}(y_1)
  \]
  in the category of towers. Because it is a relative smash product of filtered $\E_3$-algebras, it has the structure of a filtered $\E_2$-algebra.
  
  The map $\mb P^{\E_3}(x_0) \to \mb P^{\E_3}(y_1)$ is an equivalence on underlying spectra, and so the underlying $\E_2$-algebra of $R$ is the sphere spectrum $\SS$.
  
  On associated graded, this becomes a diagram of graded $\E_3$-algebras
  \[
  \SS \from \mb P_{gr}^{\E_3}(\bar x_0) \to \mb P_{gr}^{\E_3}(\bar y_1).
  \]
  The left-hand map still sends $\bar x_0$ to the element $p$. The right-hand map now sends $\bar x_0$ to $0$ because $y_1$, being of filtration $1$, now has trivial image in degree zero of the associated graded. Therefore, there is a factorization
  \[
  \mb P_{gr}^{\E_3}(\bar x_0) \to \SS \to \mb P_{gr}^{\E_3}(\bar y_1),
  \]
  and the associated graded has an equivalence
  \[
  \mathrm{gr}(R) \simeq \SS \smashover{\mb P_{gr}^{\E_3}(\bar x_0)} \mb P_{gr}^{\E_3}(\bar y_1) \simeq (\SS \smashover{\mb P_{gr}^{\E_3}(\bar x_0)} \SS)\ \smashover{\SS} \ \mb P_{gr}^{\E_3}(\bar y_1).
  \]
  Theorem~\ref{thm:e3smash} shows that this algebra $\SS \smashover{\mb P_{gr}^{\E_3}(x_0)} \SS$ is the algebra $H\F_p$ concentrated in grading zero, and so this becomes
  \[
  \mathrm{gr}(R) \simeq H\F_p \wedge \Free_{\E_3}(y_1)
  \]
  as graded $\E_2$-algebras.
\end{proof}

\begin{corollary}\label{cor:weirdmaysseq}
  The filtered spectrum of Theorem~\ref{thm:weirdmayfilt} gives rise to a spectral sequence of graded-commutative algebras, whose abutment is $\pi_\ast \mb S^\wedge_p$.

  For $p=2$, the $E_1$-term is
  \[
  \F_2[\mf h_{i,j}],
  \]
  where $\mf h_{i,j}$ are defined for $i \geq 1, j \geq 0$, with total degree and filtration given by the bidegree $(2^{i+j} - 2^j - 1, 2^{i+j-1})$. For $p$ odd, the $E_1$-term is
  \[
    \F_p[\mf v_i] \otimes \Lambda[\mf h_{i,j}] \otimes \F_p[\mf b_{i,j}],
  \]
  where the bidegree of $\mf v_i$ is $(2p^i-2, p^{i})$ for $i \geq 0$, of $\mf h_{i,j}$ is $(2p^{i+j}-2p^j-1,p^{i+j})$ for $i \geq 1, j \geq 0$, and of $\mf b_{i,j}$ is $(2p^{i+j+1}-2p^{j+1}-2,p^{i+j+1})$ for $i \geq 1, j \geq 0$.
\end{corollary}

\begin{proof}
  The coefficient ring of the associated graded is the homology of a free $\E_3$-algebra on the pointed space $S^1$, with generator $y_1$ in bidegree $(1,0)$. This is completely determined in \cite[Theorem III.3.1]{CohenLadaMayHomology} in terms of the Dyer--Lashof operations. At $p=2$, the Dyer--Lashof operation $Q_i$ sends an element in bidegree $(r,s)$ to bidegree $(2r+i, 2s)$, while at odd primes the operations $Q_i$ and $\beta Q_i$ send such an element to bidegree $(pr + 2i(p-1), ps)$ and $(pr + 2i(p-1)-1, ps)$ respectively. (Here we take the convention that lower-indexed Dyer--Lashof operations satisfy $Q_j x = Q^{|x|/2 + j}$ at odd primes).

  At $p=2$, if we define $\mf h_{i,j} = Q_1^{(j)} Q_2^{(i-1)} y_1$, then
  \[
    \pi_* \mathrm{gr}(R) \cong \F_2[\mf h_{i,j} \mid i \geq 1, j \geq 0].
  \]
  The element $\mf h_{i,j}$ is in bidegree $(2^{i+j} - 2^j - 1, 2^{i+j-1})$, which can be shown by induction.
  
  At odd primes, if we define
  \begin{align*}
    \mf v_i &= Q_1^{(i)} y_1\\
    \mf h_{i,j} &= Q_{1/2}^{(j)} \beta Q_1^{(i)} y_1\\
    \mf b_{i,j} &= \beta Q_{1/2}^{(j+1)} \beta Q_1^{(i)} y_1
  \end{align*}
  then
  \[
    \pi_* \mathrm{gr}(R) \cong \F_p[\mf v_i \mid i \geq 0] \otimes \Lambda[\mf h_{i,j} \mid i \geq 1, j \geq 0] \otimes \F_p[\mf b_{i,j} \mid i \geq 1, j \geq 0] .
  \]
  Here the bidegree of $\mf v_i$ is $(2p^i-2, p^{i})$, of $\mf h_{i,j}$ is $(2p^{i+j}-2p^j-1,p^{i+j})$, and of $\mf b_{i,j}$ is $(2p^{i+j+1}-2p^{j+1}-2,p^{i+j+1})$.
  
  These associated graded coefficient rings are then the $E_1$-terms of the associated spectral sequence.
\end{proof}

\begin{remark}\label{rmk:sseqblather}
   While the underlying graded rings are abstractly isomorphic to the $E_1$-term of the May spectral sequence, the filtration and the differentials are quite different.
   
   For example, in the $2$-primary May spectral sequence, the element $\nu \in \pi_3 \mb S$ is detected by $h_{1,2}$ and the element $\sigma \in \pi_7 \mb S$ is detected by $h_{1,3}$. In the spectral sequence of Corollary~\ref{cor:weirdmaysseq}, neither $\mf h_{1,2}$ nor $\mf h_{1,3}$ are permanent cycles; $\nu$ is instead detected by the corrected element $\mf h_{1,2} + \mf h_{1,1} \mf h_{2,0}$, while $\sigma$ is detected by $\mf h_{1,1} \mf h_{3,0} + \mf h_{2,1} \mf h_{2,0}$, in strictly lower filtration than $\mf h_{1,3}$. As a possible point of view, the Hopf invariant classes in the classical Adams spectral sequence arise from algebraic power operations that require geometric correction terms before they are genuinely represented by stable homotopy elements. In this new spectral sequence, the correction terms have lower filtration than the power operations themselves, and so are detected first.
\end{remark}

\begin{remark}
   There is an alternative approach to the construction of this filtered object $R$ using Goodwillie calculus, which will appear in forthcoming joint work of the second author with Eldred.
\end{remark}

\subsection{Dual Steenrod quotients}

The following is a specialization of Proposition~\ref{prop:freeEkpushout}.

\begin{proposition}\label{prop:ekquot}
  Suppose that $B \to C$ is a map of commutative ring spectra and $\alpha \in \pi_*(C)$ is an element. Then the quotient, in the category of $\E_k$ algebras in $\LMod_C$, has an equivalence:
  \[
    C \mathop{\aquot}_{\E_k} \alpha \simeq C \smashover{\Free^B_{\E_{k+1}}(S^n)} B.
  \]
\end{proposition}

\begin{corollary}\label{cor:killingqk}
  Suppose that $C$ is an $\E_\infty$ $B$-algebra and $R$ is an $\E_k$ $C$-algebra. If $\alpha \in \pi_* C$ maps to zero in $R$, then all of the $B$-algebra $\E_{k+1}$ Dyer-Lashof operations on $\alpha$ go to zero in $R$.
\end{corollary}

This allows us to prove the following odd-primary analogue of \cite[Theorem~1.2]{BHLSZQuotient}.

\begin{theorem}\label{thm:oddsteenrod}
  Let $p$ be an odd prime and $\mc A = H\F_p \wedge H\F_p$ be the commutative ring spectrum whose coefficient ring is the dual Steenrod algebra. Then the associative quotient $\mc A \aquot \tau_n$ has coefficient ring
  \[
    \mc A_\ast / (\xi_{n+1}, \xi_{n+2}, \dots, \tau_n, \tau_{n+1}, \dots)
  \]
  and $\mc A \aquot \bar \tau_n$ has coefficient ring
  \[
    \mc A_\ast / (\bar \xi_{n+1}, \bar \xi_{n+2}, \dots, \bar \tau_n, \bar \tau_{n+1}, \dots).
  \]
\end{theorem}

\begin{remark}
  The subring
  \[
    \F_p[\xi_1, \dots, \xi_n] \otimes \Lambda[\tau_0, \dots, \tau_{n-1}] \subset \mc A_\ast
  \]
  maps isomorphically onto both quotients. 
\end{remark}

\begin{proof}
  The conjugation operation is realized by an automorphism of $\mc A$ as a commutative ring spectrum, so it suffices to prove the case of $\bar \tau_n$.
  
  Consider the left unit map $H\F_p \to \mc A = H\F_p \wedge H\F_p$ of commutative ring spectra. Then Proposition~\ref{prop:ekquot} (with $k=1$) shows that there is a formula for the free $\E_1$-quotient:
\[
  \mc A \aquot \alpha \simeq \mc A \smashover{\Free^{H\F_p}_{\E_2}(S^n)} H\F_p.
\]
In particular, there is a K\"unneth spectral sequence
\[
  \Tor^{H_* \Free_{\E_2}(S^n)}_{**}(\mc A_*, \F_p) \Rightarrow \pi_* (\mc A \aquot \alpha).
\]
By \cite[Theorem III.3.1]{CohenLadaMayHomology}, for odd $p$ and $n$ there is an isomorphism
\[
  H_* \Free_{\E_2}(S^n) \cong \Lambda[\alpha, Q_{1/2}(\alpha), Q_{1/2}^{(2)} (\alpha), \dots] \otimes \F_p[\beta Q_{1/2} (\alpha), \beta Q_{1/2}^{(2)}(\alpha), \dots].
\]
(As in Corollary~\ref{cor:weirdmaysseq}, we take the convention that lower-indexed Dyer--Lashof operations satisfy $Q_j x = Q^{|x|/2 + j}$ at odd primes).

Suppose $\alpha = \bar \tau_{n+1}$. Using the Dyer--Lashof operations from the left unit, $Q_{1/2} \bar \tau_n = \bar \tau_{n+1}$ and $\beta Q_{1/2} \bar \tau_n = \bar \xi_{n+1}$ by \cite[Theorem III.2.3]{Hinfinity}. Therefore, the $\E_2$-term of the K\"unneth spectral sequence can be rewritten as
\[
  \Tor_{**}^{\Lambda[\bar \tau_n, \bar \tau_{n+1}, \dots] \otimes \F_p[\bar \xi_{n+1}, \bar \xi_{n+2}, \dots]} (\Lambda[\bar \tau_0, \bar \tau_{1}, \dots] \otimes \F_p[\bar \xi_1, \bar \xi_2, \dots], \F_p).
\]
However, the first $\Tor$-factor is flat over the base ring. As a result, the K\"unneth spectral sequence degenerates down to an isomorphism
\[
  \pi_* (\mc A \aquot \bar \tau_n) \cong \mc A_\ast / (\bar \tau_n, \bar \tau_{n+1}, \dots, \bar \xi_{n+1}, \bar \xi_{n+2} \dots),
\]
as desired.
\end{proof}

\printbibliography

\end{document}